\documentclass[11pt,reqno]{amsart}
\usepackage{amsmath,amssymb,latexsym,soul,cite,mathrsfs}
\usepackage{color,enumitem,graphicx}
\usepackage[colorlinks=true,urlcolor=blue,
citecolor=red,linkcolor=blue,linktocpage,pdfpagelabels,
bookmarksnumbered,bookmarksopen]{hyperref}
\usepackage[english]{babel}
\usepackage{enumitem}

\usepackage[left=2.7cm,right=2.7cm,top=2.9cm,bottom=2.9cm]{geometry}

\usepackage[hyperpageref]{backref}

\makeatletter
\newcommand{\leqnomode}{\tagsleft@true}
\newcommand{\reqnomode}{\tagsleft@false}
\makeatother

\numberwithin{equation}{section}

\newtheorem{theorem}{Theorem}[section]
\newtheorem{lemma}[theorem]{Lemma}

\newtheorem{proposition}[theorem]{Proposition}

\newtheorem{remark}[theorem]{Remark}

\renewcommand{\epsilon}{\varepsilon}

\renewcommand{\theta}{{\vartheta}}

\renewcommand{\rightarrow}{\to}

\title[Ground states for a coupled system of Schr\"odinger equations on $\mathbb{R}^{N}$]{Ground states for a linearly coupled system of Schr\"odinger equations on $\mathbb{R}^{N}$}

\author[J.M.\ do \'O]{Jo\~ao Marcos do \'O}
\author[JC. \ de Albuquerque]{Jos\'e Carlos de Albuquerque}

\address[J.M. do \'O]{Department of Mathematics,
	Federal University of Para\'{\i}ba
	\newline\indent
	58051-900, Jo\~ao Pessoa-PB, Brazil}
\email{\href{mailto:jmbo@pq.cnpq.br}{jmbo@pq.cnpq.br}}

\address[J.C. de~Albuquerque]{Institute of Mathematics and Statistics,
	Federal University of Goi\'{a}s
	\newline\indent
	74690-900, Goi\'{a}s-GO, Brazil}
\email{\href{mailto:joserre@gmail.com}{joserre@gmail.com}}


\begin{document}
	

\begin{abstract}
	We study the following class of linearly coupled Schr\"{o}dinger elliptic systems 
	$$
	\left\{
	\begin{array}{lr}
	-\Delta u+V_{1}(x)u=\mu|u|^{p-2}u+\lambda(x)v, & \quad x\in\mathbb{R}^{N},\\
	-\Delta v+V_{2}(x)v=|v|^{q-2}v+\lambda(x)u,    & x\in\mathbb{R}^{N},
	\end{array}
	\right.
	$$
	where $N\geq3$, $2<p\leq q\leq 2^{*}=2N/(N-2)$ and $\mu\geq0$. We consider nonnegative potentials periodic or asymptotically periodic which are related with the coupling term $\lambda(x)$ by the assumption $|\lambda(x)|\leq\delta\sqrt{V_{1}(x)V_{2}(x)}$, for some $0<\delta<1$. We deal with three cases: Firstly, we study the subcritical case, $2<p\leq q<2^{*}$, and we prove the existence of positive ground state for all parameter $\mu\geq0$. Secondly, we consider the critical case, $2<p<q=2^{*}$, and we prove that there exists $\mu_{0}>0$ such that the coupled system possesses positive ground state solution for all $\mu\geq\mu_{0}$. In these cases, we use a minimization method based on Nehari manifold. Finally, we consider the case $p=q=2^{*}$, and we prove that the coupled system has no positive solutions. For that matter, we use a Pohozaev identity type.
\end{abstract}
	\maketitle
	

\section{Introduction}

\noindent

We are interested in establish existence and nonexistence results for the following class of linearly coupled systems involving nonlinear Schr\"{o}dinger equations 
\begin{equation}\label{paper1j0000}
	\left\{
	\begin{array}{lr}
		-\Delta u+V_{1}(x)u=\mu|u|^{p-2}u+\lambda(x)v, & \quad x\in\mathbb{R}^{N},\\
		-\Delta v+V_{2}(x)v=|v|^{q-2}v+\lambda(x)u,    & x\in\mathbb{R}^{N},
	\end{array}
	\right.
\end{equation}
where $N\geq3$, $2<p\leq q\leq 2^*=2N/(N-2)$ is the critical Sobolev exponent. Our main goal here is to prove the existence of positive ground states for the subcritical case, that is, when $2<p\leq q<2^{*}$ and for the critical case when $2<p<q=2^{*}$. In the critical case, the existence of ground state will be related with the parameter $\mu$ introduced in the first equation. For the critical case when $p=q=2^{*}$, we make use of a Pohozaev type identity to prove that System~\eqref{paper1j0000} does not admit positive solution. We are concerned with two classes of nonnegative potentials: periodic and asymptotically periodic. Before we introduce our assumptions and the main results, we give a brief motivation to study this class of systems.


\subsection{Motivation and related results}
Solutions of System \eqref{paper1j0000} are related with solutions of the following two-component system of nonlinear Schr\"{o}dinger equations
\begin{equation}\label{paper1j00}
	\left\{
	\begin{array}{lr}
		-i\displaystyle\frac{\partial\psi}{\partial t}=\Delta \psi-V_{1}(x)\psi+\mu|\psi|^{p-2}\psi+\lambda(x)\phi, & \quad x\in \mathbb{R}^{N}, \ t\geq0,\\
		-i\displaystyle\frac{\partial\phi}{\partial t}=\Delta \phi-V_{2}(x)\phi+|\phi|^{p-2}\phi+\lambda(x)\psi,    & x\in \mathbb{R}^{N}, \ t\geq0.
	\end{array}
	\right. 
\end{equation}
Such class of systems arise in various branches of mathematical physics and nonlinear optics, see for instance \cite{fisica}. For System \eqref{paper1j00}, a solution of the form
$$
(\psi(t,x),\phi(t,x))=(\exp(-iEt)u(x),\exp(-iEt)v(x)),
$$
where $E$ is some real constant, is called \textit{standing wave solution}. Moreover, $(\psi,\phi)$ is a solution of \eqref{paper1j00} if and only if $(u,v)$ solves the following system
$$
\left\{
\begin{array}{lr}
-\Delta u+(V_{1}(x)-E)u=\mu|u|^{p-2}u+\lambda(x)v, & \quad x\in\mathbb{R}^{N},\\
-\Delta v+(V_{2}(x)-E)v=|v|^{q-2}v+\lambda(x)u,    & x\in\mathbb{R}^{N}.
\end{array}
\right.
$$
For convenience and without loss of generality, it is replaced $V_{i}(x)-E$ by $V_{i}(x)$, that is, it is shifted $E$ to $0$. Thus, it turn to consider the coupled system \eqref{paper1j0000}.

When $\lambda(x)\equiv0$, $V_{1}(x)\equiv V_{2}(x)\equiv V(x)$, $u(x)\equiv v(x)$, $\mu=1$ and $p=q$, System \eqref{paper1j0000} reduces to the scalar equation $-\Delta u+V(x)u=|u|^{p-2}u$, in $\mathbb{R}^{N}$. There are many papers that studied this class of Schr\"{o}dinger equations under many different assumptions on the potential and nonlinearity. The literature is rather extensive, see for instance \cite{1,2,5,7,r1,r2,r3,r4} and references therein.

Our work was inspired by some papers that have appeared in the recent years concerning the study of coupled systems involving nonlinear Schr\"{o}dinger equations by using variational approach. In \cite{czo1}, Z.~Chen and W.~Zou studied the existence of ground states for the following class of critical coupled system with constant potentials
\begin{equation}\label{paper1jj1}
	\left\{
	\begin{array}{lr}
		-\Delta u+\mu u=|u|^{p-2}u+\lambda v,     & \quad x\in\mathbb{R}^{N},\\
		-\Delta v+\nu v=|v|^{2^{*}-2}v+\lambda u, & x\in\mathbb{R}^{N}.
	\end{array}
	\right.
\end{equation}
They proved that there exists critical parameters $\mu_{0}>0$ and $\lambda_{\mu,\nu}\in[\sqrt{(\mu-\mu_{0})\nu},\sqrt{\mu\nu})$ such that \eqref{paper1jj1} has a positive ground state when $\lambda>\lambda_{\mu,\nu}$ and has no ground state solutions when $\mu>\mu_{0}$ and $\lambda<\lambda_{\mu,\nu}$. In \cite{czo2}, the same authors studied a class of coupled systems involving general nonlinearities in the subcritical sense. In \cite{m2}, Z.~Guo and W.~Zou obtained existence of positive ground states for another class of critical coupled systems. For more existence results concerning coupled systems we refer the readers to \cite{acr,pompo,maia,zj,t} and references therein.

Motivated by the above discussion, the current paper is concerned to study the class of coupled systems introduced by \eqref{paper1j0000} in the subcritical and critical sense. This class of systems is characterized by its lack of compactness due to the fact that the equations are defined in whole Euclidean space $\mathbb{R}^{N}$, which roughly speaking, originates from the invariance of $\mathbb{R}^{N}$ with respect to translation and dilation. Furthermore, we have the fact that \eqref{paper1j0000} involves strongly coupled Schr\"{o}dinger elliptic equations because of the linear terms in the right hand side. To overcome these difficulties, we shall use a variational approach based on Nehari manifold in combination with a lemma due to P.L. Lions (see Lemma~\ref{paper1lions}).


\subsection{Assumptions}

Firstly, we deal with the following class of coupled systems
\begin{equation}\label{paper1j0}
	\left\{
	\begin{array}{lr}
		-\Delta u+V_{1,\mathrm{o}}(x)u=\mu|u|^{p-2}u+\lambda_{\mathrm{o}}(x)v, & \quad x\in\mathbb{R}^{N},\\
		-\Delta v+V_{2,\mathrm{o}}(x)v=|v|^{q-2}v+\lambda_{\mathrm{o}}(x)u,    & x\in\mathbb{R}^{N},
	\end{array}
	\right. \tag{$S_{\mathrm{o}}^{\mu}$}
\end{equation}
where $V_{1,\mathrm{o}}(x)$, $V_{2,\mathrm{o}}(x)$ and $\lambda_{\mathrm{o}}(x)$ denote periodic functions. In view of the presence of the potentials we introduce the following space
$$
E_{i,\mathrm{o}}=\left\{u\in H^{1}(\mathbb{R}^{N}):\int_{\mathbb{R}^{N}}V_{i,\mathrm{o}}(x)u^{2}\;\mathrm{d}x<+\infty\right\}, \quad i=1,2,
$$
endowed with the inner product
$$
(u,v)_{E_{i,\mathrm{o}}}=\int_{\mathbb{R}^{N}}\nabla u\nabla v\;\mathrm{d}x+\int_{\mathbb{R}^{N}}V_{i,\mathrm{o}}(x)uv\;\mathrm{d}x,
$$
to which corresponds the induced norm $\|u\|_{E_{i,\mathrm{o}}}^{2}=(u,u)_{E_{i,\mathrm{o}}}$. In order to establish a variational approach to treat System~\eqref{paper1j0}, we need to require suitable assumptions on the potentials. For each $i = 1, 2$, we assume that

\begin{enumerate}[label=($V_{1}$),ref=$(V_{1})$] 
	\item \label{paper1A1}
	$V_{i,\mathrm{o}},\lambda_{\mathrm{o}}\in C(\mathbb{R}^{N})$ are $1$-periodic in each of $x_{1}, x_{2},..., x_{N}$.  
\end{enumerate}

\begin{enumerate}[label=($V_{2}$),ref=$(V_{2})$] 
	\item \label{paper1A2}
	$V_{i,\mathrm{o}}(x)\geq0$ for all $x\in\mathbb{R}^{N}$ and 
	$$
	\nu_{i,\mathrm{o}}=\inf_{u\in E_{i,\mathrm{o}}}\left\{\int_{\mathbb{R}^{N}}|\nabla u|^{2}\;\mathrm{d}x+\int_{\mathbb{R}^{N}}V_{i,\mathrm{o}}(x)u^{2}\;\mathrm{d}x: \int_{\mathbb{R}^{N}}u^{2}\;\mathrm{d}x=1\right\}>0.
	$$			
\end{enumerate}

\begin{enumerate}[label=($V_{3}$),ref=$(V_{3})$] 
	\item \label{paper1A3}
	$|\lambda_{\mathrm{o}}(x)|\leq\delta\sqrt{V_{1,\mathrm{o}}(x)V_{2,\mathrm{o}}(x)}$, for some $\delta\in(0,1)$, for all $x\in\mathbb{R}^{N}$.	 
\end{enumerate}

\begin{enumerate}[label=($V_{3}'$),ref=$(V_{3}')$] 
	\item \label{paper1A9}
	$0<\lambda_{\mathrm{o}}(x)\leq\delta\sqrt{V_{1,\mathrm{o}}(x)V_{2,\mathrm{o}}(x)}$, for some $\delta\in(0,1)$, for all $x\in\mathbb{R}^{N}$.	 
\end{enumerate}

We set the product space $E_{\mathrm{o}}=E_{1,\mathrm{o}}\times E_{2,\mathrm{o}}$. We have that $E_{\mathrm{o}}$ is a Hilbert space when endowed with the inner product
$$
((u,v),(w,z))_{E_{\mathrm{o}}}=(u,w)_{E_{1,\mathrm{o}}}+(v,z)_{E_{2,\mathrm{o}}},
$$
to which corresponds the induced norm 
$$
\|(u,v)\|_{E_{\mathrm{o}}}^{2}=((u,v),(u,v))_{E_{\mathrm{o}}}=\|u\|_{E_{1,\mathrm{o}}}^{2}+\|v\|_{E_{2,\mathrm{o}}}^{2}.
$$
Associated to System~\eqref{paper1j0} we have the functional $I_{\mu,\mathrm{o}}:E_{\mathrm{o}}\rightarrow\mathbb{R}$ defined by
$$
I_{\mu,\mathrm{o}}(u,v)=\frac{1}{2}\left(\|(u,v)\|_{E_{\mathrm{o}}}^{2}-2\int_{\mathbb{R}^{N}}\lambda_{\mathrm{o}}(x)uv\; \mathrm{d}x\right)-\frac{\mu}{p}\|u\|_{p}^{p}-\frac{1}{q}\|v\|_{q}^{q}.
$$
Using our assumptions we can check that $I_{\mu,\mathrm{o}}$ is well defined and is of class $C^{2}$ with derivative given by
\begin{align*}
	\langle  I_{\mu,\mathrm{o}}^{\prime}(u,v),(\phi,\psi)\rangle=((u,v),(\phi,\psi))_{E_{\mathrm{o}}}-\int_{\mathbb{R}^{N}}\left(\mu|u|^{p-2}u\phi+|v|^{q-2}v\psi+\lambda_{\mathrm{o}}(x)\left(u\psi+v\phi\right)\right)\; \mathrm{d}x,
\end{align*}
where $(\phi,\psi)\in C^{\infty}_{0}(\mathbb{R}^{N})\times C^{\infty}_{0}(\mathbb{R}^{N})$. Thus critical points of $I_{\mu,\mathrm{o}}$ correspond to weak solutions of \eqref{paper1j0} and conversely. 

We say that a solution $(u_{0},v_{0})\in E_{\mathrm{o}}$ for System \eqref{paper1j0} is a ground state (or least energy) solution if $(u_{0},v_{0})\neq(0,0)$ and its energy is minimal among the energy of all nontrivial solutions, that is, $ I_{\mu,\mathrm{o}}(u_{0},v_{0})\leq I_{\mu,\mathrm{o}}(u,v)$ for any other solution $(u,v)\in E_{\mathrm{o}}\backslash\{(0,0)\}$. We say that $(u_{0},v_{0})$ is nonnegative (nonpositive) if $u_{0},v_{0}\geq0$ ($u_{0},v_{0}\leq0$) and positive (negative) if $u_{0},v_{0}>0$ ($u_{0},v_{0}<0$) respectively.

We are also concerned with the existence of ground states for the following class of coupled systems	 
\begin{equation}\label{paper1jap}
	\left\{
	\begin{array}{lr}
		-\Delta u+V_{1}(x)u=\mu|u|^{p-2}u+\lambda(x)v, & \quad x\in\mathbb{R}^{N},\\
		-\Delta v+V_{2}(x)v=|v|^{q-2}v+\lambda(x)u, & x\in\mathbb{R}^{N},
	\end{array}
	\right. \tag{$S^{\mu}$}
\end{equation} 
when the potentials $V_{1}(x),$ $V_{2}(x)$ and $\lambda(x)$ are asymptotically periodic at infinity, that is, they are infinity limit of periodic functions $V_{1,\mathrm{o}}(x)$, $V_{2,\mathrm{o}}(x)$ and $\lambda_{\mathrm{o}}(x)$. In analogous way, we may define the suitable product space $E=E_{1}\times E_{2}$ considering the asymptotically periodic potential $V_{i}(x)$ instead $V_{i,\mathrm{o}}(x)$. In order to give a variational approach for our problem, for $i=1,2$ we assume the following hypotheses:

\begin{enumerate}[label=($V_{4}$),ref=$(V_{4})$] 
	\item \label{paper1A6}
	$V_{i},\lambda\in C(\mathbb{R}^{N})$, $V_{i}(x)<V_{i,\mathrm{o}}(x)$, $\lambda_{\mathrm{o}}(x)<\lambda(x)$, for all $x\in \mathbb{R}^{N}$ and
	$$
	\lim_{|x|\rightarrow+\infty}|V_{i,\mathrm{o}}(x)-V_{i}(x)|=0 \quad \mbox{and} \quad \lim_{|x|\rightarrow+\infty}|\lambda(x)-\lambda_{\mathrm{o}}(x)|=0.
	$$
\end{enumerate}

\begin{enumerate}[label=($V_{5}$),ref=$(V_{5})$] 
	\item \label{paper1A7}
	$V_{i}(x)\geq0$ for all $x\in\mathbb{R}^{N}$ and 
	$$
	\nu_{i}=\inf_{u\in E_{i}}\left\{\int_{\mathbb{R}^{N}}|\nabla u|^{2}\;\mathrm{d}x+\int_{\mathbb{R}^{N}}V_{i}(x)u^{2}\;\mathrm{d}x: \int_{\mathbb{R}^{N}}u^{2}\;\mathrm{d}x=1\right\}>0.
	$$		 
\end{enumerate}

\begin{enumerate}[label=($V_{6}$),ref=$(V_{6})$] 
	\item \label{paper1A8}
	$|\lambda(x)|\leq\delta\sqrt{V_{1}(x)V_{2}(x)}$, for some $\delta\in(0,1)$, for all $x\in\mathbb{R}^{N}$.
\end{enumerate}

\begin{enumerate}[label=($V_{6}'$),ref=$(V_{6}')$] 
	\item \label{paper1A10}
	$0<\lambda(x)\leq\delta\sqrt{V_{1}(x)V_{2}(x)}$, for some $\delta\in(0,1)$, for all $x\in\mathbb{R}^{N}$.
\end{enumerate}


\subsection{Statement of the main results}    

The main results of the paper are the following:

\begin{theorem}\label{paper1A}
	Assume that \ref{paper1A1}-\ref{paper1A3} hold. If $2<p\leq q<2^{*}$, then System~\eqref{paper1j0} possesses a nonnegative ground state solution $(u_{0},v_{0})\in C^{1,\beta}_{loc}(\mathbb{R}^{N})\times C^{1,\beta}_{loc}(\mathbb{R}^{N})$ for some $\beta\in(0,1)$, for all $\mu\geq0$. If \ref{paper1A9} holds, then the ground state is positive.
\end{theorem}

\begin{theorem}\label{paper1B}
	Assume that \ref{paper1A1}-\ref{paper1A3} hold. If $2<p<q=2^{*}$, then there exists $\mu_{0}>0$ such that System~\eqref{paper1j0} possesses a nonnegative ground state solution $(u_{0},v_{0})\in E_{\mathrm{o}}$, for all $\mu\geq\mu_{0}$. If \ref{paper1A9} holds, then the ground state is positive.
\end{theorem}

\begin{theorem}\label{paper1AB}
	Suppose that assumptions \ref{paper1A1}-\ref{paper1A8} hold. If $2<p\leq q<2^{*}$, then System~\eqref{paper1jap} possesses a nonnegative ground state solution $(u_{0},v_{0})\in C^{1,\beta}_{loc}(\mathbb{R}^{N})\times C^{1,\beta}_{loc}(\mathbb{R}^{N})$ for some $\beta\in(0,1)$, for all $\mu\geq0$. Moreover, if $2<p<q=2^{*}$, then there exists $\mu_{0}>0$ such that System~\eqref{paper1jap} possesses a nonnegative ground state solution for all $\mu\geq\mu_{0}$. If \ref{paper1A10} holds, then the ground states are positive.
\end{theorem} 

\begin{theorem}\label{paper1C}
	Suppose that $p=q=2^{*}$ and \ref{paper1A8} holds. In addition, for $i=1,2$ we consider the following assumptions:
	\begin{enumerate}[label=($V_7$),ref=$(V_7)$] 
		\item
		$V_{i}\in C^{1}(\mathbb{R}^{N})$ is nonnegative and $0\leq\langle\nabla V_{i}(x),x\rangle\leq CV_{i}(x)$.
		\label{paper1A4}
	\end{enumerate}
	
	\begin{enumerate}[label=($V_8$),ref=$(V_8)$] 
		\item
		$\lambda\in C^{1}(\mathbb{R}^{N})$, $|\langle\nabla\lambda(x),x\rangle|\leq C|\lambda(x)|$ and $\langle\nabla\lambda(x),x\rangle\leq0$.
		\label{paper1A5}
	\end{enumerate}
	Then, System~\eqref{paper1jap} has no positive classical solution for all $\mu\geq0$.
\end{theorem}

\begin{remark}
	A typical example of functions satisfying \ref{paper1A4} and \ref{paper1A5} is $\lambda(x)=-(1/4)\|x\|^{2}$ and $V_{i}(x)=(1/2)\|x\|^{2}$.
\end{remark}


\subsection{Notation}
Let us introduce the following notation:	
\begin{itemize}
	\item $C$, $\tilde{C}$, $C_{1}$, $C_{2}$,... denote positive constants (possibly different).
	\item $B_{R}(x_{0})$ denotes the open ball centered at $x_{0}$ and radius $R>0$.
	\item The norm in $L^{p}(\mathbb{R}^{N})$ and $L^{\infty}(\mathbb{R}^{N})$, will be denoted respectively by $\|\cdot\|_{p}$ and $\|\cdot\|_{\infty}$.
	\item $o_{n}(1)$ denotes a sequence which converges to $0$ as $n\rightarrow\infty$.
\end{itemize}


\subsection{Outline}
In the forthcoming section we introduce and give some properties of the Nehari manifold associated to \eqref{paper1j0}. In Section \ref{paper1s3}, we deal with System~\eqref{paper1j0} with subcritical growth: $2<p\leq q<2^{*}$. For this matter we use a minimization method based on Nehari manifold to get a positive ground state solution and a bootstrap argument to obtain regularity. In Section~\ref{paper1s5}, we study System~\eqref{paper1j0} with critical growth, precisely: $2<p<q=2^{*}$. In the periodic case, the key point is to use the invariance of the energy functional under translations to recover the compactness of the minimizing sequence. In Section~\ref{paper1s6}, we study the existence of ground states when the potentials are asymptotically periodic. For this purpose, we establish a relation between the energy levels associated to Systems~\eqref{paper1j0} and \eqref{paper1jap}. In Section \ref{paper1s4}, we make use of Pohozaev type identity to prove the nonexistence of positive classical solutions for System~\eqref{paper1jap} in the critical case, $p=q=2^{*}$.


\section{Preliminary results}\label{paper1s1}

In this section we provide preliminary results used throughout the paper.

\begin{lemma}\label{paper1lemma1}
	If \ref{paper1A3} holds, then we have
	\begin{equation}\label{paper1j27}
		\|(u,v)\|_{E_{\mathrm{o}}}^{2}-2\int_{\mathbb{R}^{N}}\lambda_{\mathrm{o}}(x)uv\;\mathrm{d}x\geq(1-\delta)\|(u,v)\|_{E_{\mathrm{o}}}^{2}, \quad \mbox{for all} \hspace{0,2cm} (u,v)\in E_{\mathrm{o}}.
	\end{equation} 
\end{lemma}
\begin{proof}
	For $(u,v)\in E_{\mathrm{o}}$, we have
	$$
	0\leq\left(\sqrt{V_{1,\mathrm{o}}(x)}|u|-\sqrt{V_{2,\mathrm{o}}(x)}|v|\right)^{2}=V_{1,\mathrm{o}}(x)u^{2}-2\sqrt{V_{1,\mathrm{o}}(x)}|u|\sqrt{V_{2,\mathrm{o}}(x)}|v|+V_{2,\mathrm{o}}(x)v^{2},
	$$
	which together with assumption \ref{paper1A3} implies that
	\begin{eqnarray*}
		-2\int_{\mathbb{R}^{N}}\lambda_{\mathrm{o}}(x)uv\;\mathrm{d}x & \geq & -2\delta\int_{\mathbb{R}^{N}}\sqrt{V_{1,\mathrm{o}}(x)}|u|\sqrt{V_{2,\mathrm{o}}(x)}|v|\;\mathrm{d}x\\
		& \geq & -\delta\left(\int_{\mathbb{R}^{N}}V_{1,\mathrm{o}}(x)u^{2}\;\mathrm{d}x+\int_{\mathbb{R}^{N}}V_{2,\mathrm{o}}(x)v^{2}\;\mathrm{d}x\right)\\
		& \geq & -\delta\|(u,v)\|_{E_{\mathrm{o}}}^{2},
	\end{eqnarray*}	 
	which easily implies that \eqref{paper1j27} holds.
\end{proof}


In order to prove the existence of ground states, we introduce the Nehari manifold associated to System~\eqref{paper1j0}
$$
\mathcal{N}_{\mu,\mathrm{o}}=\left\{(u,v)\in E_{\mathrm{o}}\backslash\{(0,0)\}:\langle I_{\mu,\mathrm{o}}^{\prime}(u,v),(u,v)\rangle=0\right\}.
$$
Notice that if $(u,v)\in\mathcal{N}_{\mu,\mathrm{o}}$, then
\begin{equation}\label{paper1j8}
	\|(u,v)\|_{E_{\mathrm{o}}}^{2}-2\int_{\mathbb{R}^{N}}\lambda_{\mathrm{o}}(x)uv\; \mathrm{d}x=\mu\|u\|_{p}^{p}+\|v\|_{q}^{q}.
\end{equation}

\begin{lemma}\label{paper1nehari}
	There exists $\alpha>0$ such that
	\begin{equation}\label{paper1j26}
		\|(u,v)\|_{E_{\mathrm{o}}}\geq\alpha, \quad \mbox{for all} \hspace{0,2cm} \ (u,v)\in\mathcal{N}_{\mu,\mathrm{o}}.
	\end{equation}
	Moreover, $\mathcal{N}_{\mu,\mathrm{o}}$ is a $C^{1}$-manifold.
\end{lemma}
\begin{proof}
	Let $(u,v)\in\mathcal{N}_{\mu,\mathrm{o}}$. By using \eqref{paper1j27}, \eqref{paper1j8} and Sobolev embedding, we deduce that
	$$
	(1-\delta)\|(u,v)\|_{E_{\mathrm{o}}}^{2} \leq  \|(u,v)\|_{E_{\mathrm{o}}}^{2}-2\int_{\mathbb{R}^{N}}\lambda_{\mathrm{o}}(x)uv\; \mathrm{d}x \leq  C\left(\|(u,v)\|_{E_{\mathrm{o}}}^{p}+\|(u,v)\|_{E_{\mathrm{o}}}^{q}\right).
	$$
	Hence, we have that
	$$
	0<\frac{1-\delta}{C}\leq\|(u,v)\|_{E_{\mathrm{o}}}^{p-2}+\|(u,v)\|_{E_{\mathrm{o}}}^{q-2},
	$$
	which implies that \eqref{paper1j26} holds. Now, let $ J_{\mu,\mathrm{o}}:E_{\mathrm{o}}\backslash\{(0,0)\}\rightarrow\mathbb{R}$ be the $C^{1}$-functional defined by
	$$
	J_{\mu,\mathrm{o}}(u,v)=\langle I_{\mu,\mathrm{o}}^{\prime}(u,v),(u,v)\rangle=\|(u,v)\|_{E_{\mathrm{o}}}^{2}-2\int_{\mathbb{R}^{N}}\lambda_{\mathrm{o}}(x)uv\; \mathrm{d}x-\mu\|u\|_{p}^{p}-\|v\|_{q}^{q}.
	$$
	Notice that $\mathcal{N}_{\mu,\mathrm{o}}= J_{\mu,\mathrm{o}}^{-1}(0)$. If $(u,v)\in\mathcal{N}_{\mu,\mathrm{o}}$, then it follows from \eqref{paper1j8} that
	\begin{eqnarray*}
		\langle J_{\mu,\mathrm{o}}^{\prime}(u,v),(u,v)\rangle & = & 2\left(\|(u,v)\|^{2}_{E_{\mathrm{o}}}-2\int_{\mathbb{R}^{N}}\lambda_{\mathrm{o}}(x)uv\;\mathrm{d}x\right)-\mu p\|u\|_{p}^{p}-q\|v\|_{q}^{q}\\
		& = & (2-p)\left(\|(u,v)\|_{E_{\mathrm{o}}}^{2}-2\int_{\mathbb{R}^{N}}\lambda_{\mathrm{o}}(x)uv\; \mathrm{d}x\right)+(p-q)\|v\|^{q}_{q},
	\end{eqnarray*}
	which together with \eqref{paper1j27}, \eqref{paper1j26} and the fact that $2<p\leq q$ implies that
	\begin{eqnarray}\label{paper1j41}
		\langle J_{\mu,\mathrm{o}}^{\prime}(u,v),(u,v)\rangle\leq (2-p)(1-\delta)\|(u,v)\|_{E_{\mathrm{o}}}^{2}\leq (2-p)(1-\delta)\alpha<0.
	\end{eqnarray}
	Thus, $0$ is a regular value of $ J_{\mu,\mathrm{o}}$ and therefore $\mathcal{N}_{\mu,\mathrm{o}}$ is a $C^{1}$-manifold.
\end{proof}

\begin{remark}
	If $(u_{0},v_{0})\in\mathcal{N}_{\mu,\mathrm{o}}$ is a critical point of $I_{\mu,\mathrm{o}}\mid_{\mathcal{N}_{\mu,\mathrm{o}}}$, then $I_{\mu,\mathrm{o}}^{\prime}(u_{0},v_{0})=0$. In fact, notice that
	$
	I_{\mu,\mathrm{o}}^{\prime}(u_{0},v_{0})=\eta J_{\mu,\mathrm{o}}^{\prime}(u_{0},v_{0}),
	$
	where $\eta\in\mathbb{R}$ is the corresponding Lagrange multiplier. Taking the scalar product with $(u_{0},v_{0})$ and using \eqref{paper1j41} we conclude that $\eta=0$.
\end{remark}

\begin{lemma}\label{paper1p1}
	Assume \ref{paper1A3} holds. Thus, for any $(u,v)\in E_{\mathrm{o}}\backslash\{(0,0)\}$, there exists a unique $t_{\mu}>0$, depending on $\mu$ and $(u,v)$, such that 
	$$
	(t_{\mu}u,t_{\mu}v)\in\mathcal{N}_{\mu,\mathrm{o}} \quad \mbox{and} \quad I_{\mu,\mathrm{o}}(t_{\mu}u,t_{\mu}v)=\max_{t\geq0} I_{\mu,\mathrm{o}}(tu,tv).
	$$
\end{lemma}
\begin{proof}
	Let $(u,v)\in E_{\mathrm{o}}\backslash\{(0,0)\}$ be fixed and consider the function $g:[0,\infty)\rightarrow\mathbb{R}$ defined by $g(t)= I_{\mu,\mathrm{o}}(tu,tv)$. Notice that
	$
	\langle I_{\mu,\mathrm{o}}^{\prime}(tu,tv),(tu,tv)\rangle=tg^{\prime}(t).
	$
	Therefore, $t_{\mu}$ is a positive critical point of $g$ if and only if $(t_{\mu}u,t_{\mu}v)\in\mathcal{N}_{\mu,\mathrm{o}}$. It follows from assumption \ref{paper1A3} that
	$$
	\|(u,v)\|_{E_{\mathrm{o}}}^{2}-2\int_{\mathbb{R}^{N}}\lambda_{\mathrm{o}}(x)uv\; \mathrm{d}x\geq0, \quad \mbox{for all} \hspace{0,2cm} (u,v)\in E_{\mathrm{o}}.
	$$
	Since $2<p\leq q$ and
	$$
	g(t)=\frac{t^{2}}{2}\left(\|(u,v)\|_{E_{\mathrm{o}}}^{2}-2\int_{\mathbb{R}^{N}}\lambda_{\mathrm{o}}(x)uv\; \mathrm{d}x\right)-\frac{t^{p}}{p}\mu\|u\|^{p}_{p}-\frac{t^{q}}{q}\|v\|_{q}^{q},
	$$
	we conclude that $g(t)<0$ for $t>0$ sufficiently large. On the other hand, by using \ref{paper1A3} and Sobolev embeddings, we have that
	\begin{eqnarray*}
		g(t) & \geq & (1-\delta)\frac{t^{2}}{2}\|(u,v)\|_{E_{\mathrm{o}}}^{2}-C_{1}\mu\frac{t^{p}}{p}\|u\|_{E_{1,\mathrm{o}}}^{p}-C_{2}\frac{t^{q}}{q}\|v\|_{E_{2,\mathrm{o}}}^{q}\\
		& \geq & t^{2}\|(u,v)\|_{E_{\mathrm{o}}}^{2}\left(\frac{1-\delta}{2}-C_{1}\mu\frac{t^{p-2}}{p}\|(u,v)\|_{E_{\mathrm{o}}}^{p-2}-C_{2}\frac{t^{q-2}}{q}\|(u,v)\|_{E_{\mathrm{o}}}^{q-2}\right)>0,
	\end{eqnarray*}
	provided $t>0$ is sufficiently small. Thus $g$ has maximum points in $(0,\infty)$. Suppose that there exists $t_{1},t_{2}>0$ with $t_{1}<t_{2}$ such that $g^{\prime}(t_{1})=g^{\prime}(t_{2})=0$. Since every critical point of $g$ satisfies
	\begin{equation}\label{jm1}
		\|(u,v)\|_{E_{\mathrm{o}}}^{2}-2\int_{\mathbb{R}^{N}}\lambda_{\mathrm{o}}(x)uv\; \mathrm{d}x=t^{p-2}\mu\|u\|_{p}^{p}+t^{q-2}\|v\|^{q}_{q},
	\end{equation}
	we have that
	$
	(t_{1}^{p-2}-t_{2}^{p-2})\mu\|u\|_{p}^{p}+(t_{1}^{q-2}-t_{2}^{q-2})\|v\|_{q}^{q}=0.
	$
	Thus $u=v=0$ which is impossible and the proof is complete. 
\end{proof}


Let us define the Nehari energy level associated with System~\eqref{paper1j0}
$$
c_{\mathcal{N}_{\mu,\mathrm{o}}}=\inf_{(u,v)\in\mathcal{N}_{\mu,\mathrm{o}}} I_{\mu,\mathrm{o}}(u,v).
$$
We claim that $c_{\mathcal{N}_{\mu,\mathrm{o}}}$ is positive. In fact, for any $(u,v)\in\mathcal{N}_{\mu,\mathrm{o}}$ we can deduce that
$$
I_{\mu,\mathrm{o}}(u,v)=\left(\frac{1}{2}-\frac{1}{p}\right)\left(\|(u,v)\|_{E_{\mathrm{o}}}^{2}-2\int_{\mathbb{R}^{N}}\lambda_{\mathrm{o}}(x)uv\; \mathrm{d}x\right)+\left(\frac{1}{p}-\frac{1}{q}\right)\|v\|^{q}_{q}.
$$
Since $2<p\leq q$, it follows from \eqref{paper1j27} and \eqref{paper1j26} that
$$
I_{\mu,\mathrm{o}}(u,v)\geq\left(\frac{1}{2}-\frac{1}{p}\right)(1-\delta)\|(u,v)\|_{E_{\mathrm{o}}}^{2}\geq\left(\frac{1}{2}-\frac{1}{p}\right)(1-\delta)\alpha>0.
$$	

\begin{remark}
	Although we used the notation for periodic functions, all results of this section remain true for asymptotically periodic functions.
\end{remark}

\section{Proof of Theorem \ref{paper1A}}\label{paper1s3}

We can use Ekeland's variational principle (see \cite{ekeland}) to obtain a sequence $(u_{n},v_{n})_{n}\subset\mathcal{N}_{\mu,\mathrm{o}}$ such that
\begin{equation}\label{paper1j6}
	I_{\mu,\mathrm{o}}(u_{n},v_{n})\rightarrow c_{\mathcal{N}_{\mu,\mathrm{o}}} \hspace{0,3cm} \mbox{and} \hspace{0,3cm}  I_{\mu,\mathrm{o}}^{\prime}(u_{n},v_{n})\rightarrow0.
\end{equation}
Notice that $(u_{n},v_{n})_n$ is bounded. In fact, recalling that $p\leq q$ it follows from \eqref{paper1j27} and \eqref{paper1j8} that
\begin{eqnarray*}
	I_{\mu,\mathrm{o}}(u_{n},v_{n}) & = & \left(\frac{1}{2}-\frac{1}{p}\right)\left(\|(u_{n},v_{n})\|_{E_{\mathrm{o}}}^{2}-2\int_{\mathbb{R}^{N}}\lambda_{\mathrm{o}}(x)u_nv_n\; \mathrm{d}x\right)+\left(\frac{1}{p}-\frac{1}{q}\right)\|v_{n}\|_{q}^{q}\\
	& \geq & \left(\frac{1}{2}-\frac{1}{q}\right)\left(1-\delta\right)\|(u_{n},v_{n})\|_{E_{\mathrm{o}}}^{2}.
\end{eqnarray*}
Since $ I_{\mu,\mathrm{o}}(u_{n},v_{n})$ is bounded, we conclude that $(u_{n},v_{n})_{n}$ is bounded in $E_{\mathrm{o}}$. Passing $(u_{n},v_{n})_{n}$ to a subsequence, we way assume that $(u_{n},v_{n})\rightharpoonup (u_{0},v_{0})$ weakly in $E_{\mathrm{o}}$. By a standard argument, we have that $I_{\mu,\mathrm{o}}^{\prime}(u_{0},v_{0})=0$. We recall the following result due to P.L.~Lions \cite[Lemma~1.21]{will} (see also \cite{lionss}).

\begin{lemma}\label{paper1lions}
	Let $r>0$ and $2\leq s<2^{*}$. If $(u_{n})_{n}\subset H^{1}(\mathbb{R}^{N})$ is a bounded sequence such that
	$$
	\lim_{n\rightarrow+\infty}\sup_{y\in\mathbb{R}^{N}}\int_{B_{r}(y)}|u_{n}|^{s}\;\mathrm{d}x=0,
	$$
	then $u_{n}\rightarrow0$ in $L^{s}(\mathbb{R}^{N})$ for $2<s<2^{*}$.
\end{lemma}

\begin{proposition}\label{paper1p3}
	There exists a ground state solution for System~\eqref{paper1j0}.
\end{proposition}
\begin{proof}
	We split the argument into two cases.
	
	\noindent\textbf{Case 1.} $(u_{0},v_{0})\neq(0,0)$.
	
	In this case, $(u_{0},v_{0})$ is a nontrivial critical point of the energy functional $I_{\mu,\mathrm{o}}$. Thus, $(u_{0},v_{0})\in\mathcal{N}_{\mu,\mathrm{o}}$. It remains to prove that $I_{\mu,\mathrm{o}}(u_{0},v_{0})=c_{\mathcal{N}_{\mu,\mathrm{o}}}$. It is clear that $c_{\mathcal{N}_{\mu,\mathrm{o}}}\leq I_{\mu,\mathrm{o}}(u_{0},v_{0})$. On the other hand, by using the semicontinuity of norm, we can deduce that
	\begin{eqnarray*}
		c_{\mathcal{N}_{\mu,\mathrm{o}}}+o_{n}(1) & = &  I_{\mu,\mathrm{o}}(u_{n},v_{n})-\frac{1}{2}\langle  I_{\mu,\mathrm{o}}^{\prime}(u_{n},v_{n}),(u_{n},v_{n})\rangle\\
		& = & \left(\frac{1}{2}-\frac{1}{p}\right)\mu\|u_{n}\|_{p}^{p}+\left(\frac{1}{2}-\frac{1}{q}\right)\|v_{n}\|_{q}^{q}\\
		& \geq & \left(\frac{1}{2}-\frac{1}{p}\right)\mu\|u_{0}\|_{p}^{p}+\left(\frac{1}{2}-\frac{1}{q}\right)\|v_{0}\|_{q}^{q}+o_{n}(1)\\
		& = &  I_{\mu,\mathrm{o}}(u_{0},v_{0})-\frac{1}{2}\langle  I_{\mu,\mathrm{o}}^{\prime}(u_{0},v_{0}),(u_{0},v_{0})\rangle+o_{n}(1)\\
		& = &  I_{\mu,\mathrm{o}}(u_{0},v_{0})+o_{n}(1),
	\end{eqnarray*}  
	which implies that $c_{\mathcal{N}_{\mu,\mathrm{o}}}\geq I_{\mu,\mathrm{o}}(u_{0},v_{0})$. Therefore, $I_{\mu,\mathrm{o}}(u_{0},v_{0})=c_{\mathcal{N}_{\mu,\mathrm{o}}}$.
	
	\noindent\textbf{Case 2.} $(u_{0},v_{0})=(0,0)$.
	
	We claim that there exist a sequence $(y_{n})_{n}\subset\mathbb{R}^{N}$ and constants $R,\xi>0$ such that
	\begin{equation}\label{paper1j50}
		\liminf_{n\rightarrow\infty}\int_{B_{R}(y_{n})}(u_{n}^{2}+v_{n}^{2})\;\mathrm{d}x\geq\xi>0.
	\end{equation}
	Suppose by contradiction that \eqref{paper1j50} does not hold. Thus, for any $R>0$ we have 
	$$
	\lim_{n\rightarrow\infty}\sup_{y\in\mathbb{R}^{N}}\int_{B_{R}(y)}(u_{n}^{2}+v_{n}^{2})\;\mathrm{d}x=0.
	$$
	It follows from Lemma~\ref{paper1lions} that $u_{n}\rightarrow0$ strongly in $L^{p}(\mathbb{R}^{N})$ and $v_{n}\rightarrow0$ strongly $L^{q}(\mathbb{R}^{N})$, for $2<p,q<2^{*}$. Since $(u_{n},v_{n})_{n}\subset\mathcal{N}_{\mu,\mathrm{o}}$, we can deduce that
	$$
	0<(1-\delta)\alpha\leq (1-\delta)\|(u_{n},v_{n})\|_{E_{\mathrm{o}}}^{2}\leq \mu\|u_{n}\|_{p}^{p}+\|v_{n}\|_{q}^{q}\rightarrow0,
	$$
	which implies that $(u_{n},v_{n})\rightarrow0$ strongly in $E_{\mathrm{o}}$. But this is impossible, since $I_{\mu,\mathrm{o}}$ is continuous and $I_{\mu,\mathrm{o}}(u_{n},v_{n})\rightarrow c_{\mathcal{N}_{\mu,\mathrm{o}}}>0$. Therefore, \eqref{paper1j50} holds. 
	
	We may assume without loss of generality that $(y_{n})_{n}\subset\mathbb{Z}^{N}$. Let us consider the shift sequence $(\tilde{u}_{n}(x),\tilde{v}_{n}(x))=(u_{n}(x+y_{n}),v_{n}(x+y_{n}))$. Since $V_{1,\mathrm{o}}(\cdot)$, $V_{2,\mathrm{o}}(\cdot)$ and $\lambda_{\mathrm{o}}(\cdot)$ are $1$-periodic functions, it follows that the energy functional $I_{\mu,\mathrm{o}}$ is invariant under translations of the form $(u,v)\mapsto (u(\cdot-z),v(\cdot-z))$ with $z\in\mathbb{Z}^{N}$. By a careful computation we can deduce that
	$$
	\|(u_{n},v_{n})\|_{E_{\mathrm{o}}}=\|(\tilde{u}_{n},\tilde{v}_{n})\|_{E_{\mathrm{o}}}, \quad I_{\mu,\mathrm{o}}(u_{n},v_{n})=I_{\mu,\mathrm{o}}(\tilde{u}_{n},\tilde{v}_{n})\rightarrow c_{\mathcal{N}_{\mu,\mathrm{o}}} \quad \mbox{and} \quad I_{\mu,\mathrm{o}}^{\prime}(\tilde{u}_{n},\tilde{v}_{n})\rightarrow0.
	$$
	Moreover, arguing as before, we can conclude that $(\tilde{u}_{n},\tilde{v}_{n})_{n}$ is a bounded sequence in $E_{\mathrm{o}}$. In this way, there exists a critical point $(\tilde{u},\tilde{v})$ of $I_{\mu,\mathrm{o}}$, such that, up to a subsequence, $(\tilde{u}_{n},\tilde{v}_{n})\rightharpoonup(\tilde{u},\tilde{v})$ weakly in $E_{\mathrm{o}}$ and $(\tilde{u}_{n},\tilde{v}_{n})\rightarrow(\tilde{u},\tilde{v})$ strongly in $L^{2}(B_{R}(0))\times L^{2}(B_{R}(0))$. Thus, using \eqref{paper1j50} we obtain
	$$
	\int_{B_{R}(0)}(\tilde{u}^{2}+\tilde{v}^{2})\;\mathrm{d}x=\liminf_{n\rightarrow\infty}\int_{B_{R}(0)}(\tilde{u}_{n}^{2}+\tilde{v}_{n}^{2})\;\mathrm{d}x=\liminf_{n\rightarrow\infty}\int_{B_{R}(y_{n})}(u_{n}^{2}+v_{n}^{2})\;\mathrm{d}x\geq\xi>0.
	$$
	Therefore, $\tilde{u}\not\equiv0$ or $\tilde{v}\not\equiv0$. The conclusion follows as in the \textbf{Case 1}.
\end{proof}

\begin{proposition}\label{paper1p9}
	There exists a nonnegative ground state solution $(\tilde{u},\tilde{v})\in C^{1,\beta}_{loc}(\mathbb{R}^{N})\times C^{1,\beta}_{loc}(\mathbb{R}^{N})$ for System~\eqref{paper1j0}, for some $\beta\in(0,1)$.
\end{proposition}

\begin{proof}
	Let $(u_{0},v_{0})\in\mathcal{N}_{\mu,\mathrm{o}}$ be the ground state obtained in the proposition~\ref{paper1p3}. From Lemma \ref{paper1p1}, there exists $t_{\mu}>0$ such that $(t_{\mu}|u_{0}|,t_{\mu}|v_{0}|)\in\mathcal{N}_{\mu,\mathrm{o}}$. Thus, we have that
	$$
	I_{\mu,\mathrm{o}}(t_{\mu}|u_{0}|,t_{\mu}|v_{0}|)\leq I_{\mu,\mathrm{o}}(t_{\mu}u_{0},t_{\mu}v_{0})\leq \max_{t\geq0}I_{\mu,\mathrm{o}}(tu_{0},tv_{0})= I_{\mu,\mathrm{o}}(u_{0},v_{0})=c_{\mathcal{N}_{\mu,\mathrm{o}}},
	$$
	which implies that $(t_{\mu}|u_{0}|,t_{\mu}|v_{0}|)$ is also a minimizer of $ I_{\mu,\mathrm{o}}$ on $\mathcal{N}_{\mu,\mathrm{o}}$. Therefore, $(t_{\mu}|u_{0}|,t_{\mu}|v_{0}|)$ is a nonnegative ground state solution for System~\eqref{paper1j0}. 
	
	To prove the regularity, we use the standard bootstrap argument. We denote $(\tilde{u},\tilde{v})=(t_{\mu}|u_{0}|,t_{\mu}|v_{0}|)$ and we define
	$$
	p_{1}(x)=\mu|\tilde{u}|^{p-2}\tilde{u}+\lambda_{\mathrm{o}}(x)\tilde{v}-V_{1,\mathrm{o}}(x)\tilde{u} \quad \mbox{and} \quad p_{2}(x)=|\tilde{v}|^{q-2}\tilde{v}+\lambda_{\mathrm{o}}(x)\tilde{u}-V_{2,\mathrm{o}}(x)\tilde{v}.
	$$
	Thus, $(\tilde{u},\tilde{v})$ is a weak solution of the restricted problem
	\begin{equation}\label{paper1j37}
		\left\{
		\begin{array}{lr}
			-\Delta \tilde{u}=p_{1}(x), & \quad x\in B_{1}(0),\\
			-\Delta \tilde{v}=p_{2}(x), & x\in B_{1}(0).
		\end{array}
		\right.
	\end{equation} 
	Using Sobolev embedding we have that $V_{1,\mathrm{o}}\tilde{u}, V_{2,\mathrm{o}}\tilde{v},\lambda_{\mathrm{o}}\tilde{u}, \lambda_{\mathrm{o}}\tilde{v}\in L^{2^{*}}(B_{1}(0))$. Moreover, $|\tilde{u}|^{p-2}\tilde{u}\in L^{r}(B_{1}(0))$ for all $1\leq r\leq 2^{*}/(p-1)$ and $|\tilde{v}|^{q-2}\tilde{v}\in L^{s}(B_{1}(0))$ for all $1\leq s\leq 2^{*}/(q-1)$.
	Let us define $r_{1}=2^{*}/(q-1)$. Since $p\leq q$, it follows that $r_{1}\leq 2^{*}/(p-1)$. Hence $|\tilde{u}|^{p-2}\tilde{u}\in L^{r_{1}}(B_{1}(0))$. Therefore, $p_{1},p_{2}\in L^{r_{1}}(B_{1}(0))$. On the other hand, for each $i=1,2$ let $w_{i}$ be the Newtonian potential of $p_{i}(x)$. Thus, in light of \cite[Theorem 9.9]{trud} we have $w_{i}\in W^{2,r_{1}}(B_{1}(0))$ and
	\begin{equation}\label{paper1j38}
		\left\{
		\begin{array}{lr}
			\Delta w_{1}=p_{1}(x), & \quad x\in B_{1}(0),\\
			\Delta w_{2}=p_{2}(x), & x\in B_{1}(0).
		\end{array}
		\right.
	\end{equation}  
	Therefore, $(\tilde{u}-w_{1},\tilde{v}-w_{2})\in H^{1}(B_{1}(0))\times H^{1}(B_{1}(0))$ is a weak solution of the problem
	$$
	\left\{
	\begin{array}{lr}
	\Delta z_{1}=0, & \quad \mbox{in} \ B_{1}(0),\\
	\Delta z_{2}=0, & \mbox{in} \ B_{1}(0).
	\end{array}
	\right.
	$$
	In view of \cite[Corollary 1.2.1]{jost}, we have that $(\tilde{u}-w_{1},\tilde{v}-w_{2})\in C^{\infty}(B_{1}(0))\times C^{\infty}(B_{1}(0))$. Therefore, $(\tilde{u},\tilde{v})\in W^{2,r_{1}}(B_{1}(0))\times W^{2,r_{1}}(B_{1}(0))$. Since $q-1<2^{*}-1$, there exists $\delta>0$ such that $(q-1)(1+\delta)=2^{*}-1$. Thus, one has
	\begin{equation}\label{paper1j40}
		r_{1}=\frac{2^{*}}{q-1}=2^{*}\frac{(1+\delta)}{2^{*}-1}=\frac{2N}{N+2}(1+\delta).
	\end{equation}
	Recall the Sobolev embedding $W^{2,r_{1}}(B_{1}(0))\hookrightarrow L^{s_{1}}(B_{1}(0))$ with $s_{1}=Nr_{1}/(N-2r_{1})$. We claim that there exists $r_{2}\in(r_{1},s_{1})$ such that $(\tilde{u},\tilde{v})\in W^{2,r_{2}}(B_{1}(0))\times W^{2,r_{2}}(B_{1}(0))$. Indeed, we define $r_{2}=s_{1}/(q-1)$ and we note that $r_{2}<s_{1}$. By using \eqref{paper1j40} we deduce that
	$$
	\frac{r_{2}}{r_{1}}=\frac{Nr_{1}}{(q-1)(N-2r_{1})r_{1}}=\frac{(N-2)(1+\delta)}{N-2-4\delta}>1+\delta,
	$$
	which implies that $r_{2}\in(r_{1},s_{1})$. By Sobolev embedding, we have
	$$
	W^{2,r_{1}}(B_{1}(0))\hookrightarrow L^{s_{1}}(B_{1}(0))\hookrightarrow L^{r_{2}}(B_{1}(0)).
	$$
	Hence, $p_{1}(x),p_{2}(x)\in L^{r_{2}}(B_{1}(0))$. From the same argument used before, we can conclude that $(\tilde{u},\tilde{v})\in W^{2,r_{2}}(B_{1}(0))\times W^{2,r_{2}}(B_{1}(0))$. Iterating, we obtain the following sequence
	$$
	r_{n+1}=\frac{1}{q-1}\left(\frac{Nr_{n}}{N-2r_{n}}\right).
	$$
	Notice that $r_{n+1}\rightarrow\infty$, as $n\rightarrow\infty$. Therefore,
	$$
	(\tilde{u},\tilde{v})\in W^{2,r}_{loc}(\mathbb{R}^{N})\times W^{2,r}_{loc}(\mathbb{R}^{N}), \quad \mbox{for all}  \hspace{0,2cm} 2\leq r<\infty.
	$$
	From Sobolev embedding, we have that $(\tilde{u},\tilde{v})\in C^{1,\beta}(B_{1}(0))\times C^{1,\beta}(B_{1}(0))$, for some $\beta\in(0,1)$.
\end{proof}

\begin{proposition}\label{paper1p11}
	If \ref{paper1A9} holds, then the ground state is positive.
\end{proposition}
\begin{proof}
	Let $(\tilde{u},\tilde{v})\in E_{\mathrm{o}}\backslash\{(0,0)\}$ be the nonnegative ground state obtained in Proposition \ref{paper1p9}. Since $(\tilde{u},\tilde{v})\neq(0,0)$ we may assume without loss of generality that $\tilde{u}\neq0$. We claim that $\tilde{v}\neq0$. In fact, arguing by contradiction we suppose that $\tilde{v}=0$. Thus,
	$$
	0=\langle I_{\mu,\mathrm{o}}^{\prime}(\tilde{u},0),(0,\psi)\rangle=-\int_{\mathbb{R}^{N}}\lambda_{\mathrm{o}}(x)\tilde{u}\psi\;\mathrm{d}x, \quad \mbox{for all} \hspace{0,2cm} \psi\in C^{\infty}_{0}(\mathbb{R}^{N}).
	$$
	Since $\lambda_{\mathrm{o}}(x)$ is positive, we have that $\tilde{u}=0$ which is a contradiction. Therefore, $\tilde{v}\neq0$. 
	
	Taking $(\varphi,0)$ as test function one sees that
	$$
	\int_{\mathbb{R}^{N}}\nabla\tilde{u}\nabla\varphi\;\mathrm{d}x+\int_{\mathbb{R}^{N}}V_{1,\mathrm{o}}(x)\tilde{u}\varphi\;\mathrm{d}x=\mu\int_{\mathbb{R}^{N}}|\tilde{u}|^{p-2}\tilde{u}\varphi\;\mathrm{d}x+\int_{\mathbb{R}^{N}}\lambda_{\mathrm{o}}(x)\tilde{v}\varphi\;\mathrm{d}x\geq0,
	$$
	for all $\varphi\geq0$, $\varphi\in C^{\infty}_{0}(\mathbb{R}^{N})$. Thus, we can deduce that
	$$
	\int_{\mathbb{R}^{N}}\nabla(-\tilde{u})\nabla\varphi\;\mathrm{d}x-\int_{\mathbb{R}^{N}}\left[-V_{1,\mathrm{o}}(x)\right](-\tilde{u})\varphi\;\mathrm{d}x\leq0, \quad \mbox{for all} \hspace{0,2cm} \varphi\geq0, \ \varphi\in C^{\infty}_{0}(\mathbb{R}^{N}).
	$$
	Moreover, since $V_{1,\mathrm{o}}(x)\geq0$ for all $x\in\mathbb{R}^{N}$, it follows that
	$$
	-\int_{\mathbb{R}^{N}}V_{1,\mathrm{o}}(x)\varphi\;\mathrm{d}x\leq0, \quad \mbox{for all} \hspace{0,2cm} \varphi\geq0, \ \varphi\in C^{\infty}_{0}(\mathbb{R}^{N}).
	$$
	In order to prove that $(\tilde{u},\tilde{v})$ is positive, we suppose by contradiction that there exists $p\in\mathbb{R}^{N}$ such that $\tilde{u}(p)=0$. Thus, since $-\tilde{u}\leq0$ in $\mathbb{R}^{N}$, for any $R>>R_{0}>0$ we have that
	$$
	0=\sup_{B_{R_{0}}(p)}(-\tilde{u})=\sup_{B_{R}(p)}(-\tilde{u}).
	$$
	By the Strong Maximum Principle \cite[Theorem~8.19]{trud} we conclude that $-\tilde{u}\equiv0$ in $B_{R}(p)$, for all $R>R_{0}$. Therefore, $\tilde{u}\equiv0$ in $\mathbb{R}^{N}$ which is a contradiction. Therefore $\tilde{u}>0$ in $\mathbb{R}^{N}$. Analogously we can prove that $\tilde{v}>0$ in $\mathbb{R}^{N}$. Therefore, the ground state $(\tilde{u},\tilde{v})$ is positive.
\end{proof}	

Theorem~\ref{paper1A} follows from Propositions \ref{paper1p3}, \ref{paper1p9} and \ref{paper1p11}.





\section{Proof of Theorem \ref{paper1B}}\label{paper1s5}

In this section, we deal with System~\eqref{paper1j0} when $2<p<q=2^{*}$. Analogously to Theorem~\ref{paper1A}, we have a sequence $(u_{n},v_{n})_{n}\subset\mathcal{N}$ satisfying \eqref{paper1j6}. Moreover, the sequence is bounded and $(u_{n},v_{n})\rightharpoonup (u_{0},v_{0})$ weakly in $E$. We have also that $(u_{0},v_{0})$ is a critical point of the energy functional $I$. We denote by $S$ the sharp constant of the embedding $D^{1,2}(\mathbb{R}^{N})\hookrightarrow L^{2^{*}}(\mathbb{R}^{N})$
\begin{equation}\label{sharp}
	S\left(\int_{\mathbb{R}^{N}}|u|^{2^{*}}\;\mathrm{d}x\right)^{2/2^{*}}\leq \int_{\mathbb{R}^{N}}|\nabla u|^{2}\;\mathrm{d}x,
\end{equation}
where $D^{1,2}(\mathbb{R}^{N}):=\{u\in L^{2^{*}}(\mathbb{R}^{N}):|\nabla u|\in L^{2}(\mathbb{R}^{N})\}$. In order to get a nontrivial critical point for $I_{\mu,\mathrm{o}}$ we need the following lemma:

\begin{lemma}\label{paper1mu}
	There exists $\mu_{0}>0$ such that $c_{\mathcal{N}_{\mu,\mathrm{o}}}<\frac{1}{N}S^{N/2}$, for all $\mu\geq\mu_{0}$. 
\end{lemma}
\begin{proof}
	Let us consider $(u,v)\in E_{\mathrm{o}}$ such that $u,v\geq0$ and $u,v\not\equiv0$. It follows from Lemma~\ref{paper1p1} that there exists a unique $t_{\mu}>0$, depending on $\mu>0$ and $(u,v)$, such that $(t_{\mu}u,t_{\mu}v)\in\mathcal{N}_{\mu,\mathrm{o}}$. Thus, by using relation \eqref{jm1} we can conclude that $t_{\mu}\rightarrow0$ as $\mu\rightarrow+\infty$. Moreover, we have that
	$$
	c_{\mathcal{N}_{\mu,\mathrm{o}}}\leq I_{\mu,\mathrm{o}}(t_{\mu}u,t_{\mu}v)\leq\frac{t_{\mu}^{2}}{2}\left(\|(u,v)\|_{E_{\mathrm{o}}}^{2}-2\int_{\mathbb{R}^{N}}\lambda_{\mathrm{o}}(x)uv\;\mathrm{d}x\right),
	$$
	and the right hand side goes to zero as $\mu$ goes to infinity. Therefore, there exists $\mu_{0}>0$ such that $c_{\mathcal{N}_{\mu,\mathrm{o}}}<\frac{1}{N}S^{N/2}$, for all $\mu\geq\mu_{0}$.
\end{proof}

In analogous way to the proof of Theorem~\ref{paper1A}, we split the proof into two cases.

\noindent\textbf{Case 1} $(u_{0},v_{0})\neq(0,0)$.

This case is completely similar to the proof of the subcritical case.

\noindent\textbf{Case 2} $(u_{0},v_{0})=(0,0)$.

Let $\mu_{0}>0$ be the parameter obtained in the Lemma~\ref{paper1mu}. We claim that if $\mu\geq\mu_{0}$, then there exists a sequence $(y_{n})_{n}\subset\mathbb{R}^{N}$ and constants $R,\xi>0$ such that
\begin{equation}\label{paper1j500}
	\liminf_{n\rightarrow\infty}\int_{B_{R}(y_{n})}(u_{n}^{2}+v_{n}^{2})\;\mathrm{d}x\geq\xi>0.
\end{equation}
In fact, suppose that \eqref{paper1j500} does not hold. Thus, for any $R>0$ we have
$$
\lim_{n\rightarrow\infty}\sup_{y\in\mathbb{R}^{N}}\int_{B_{R}(y)}(u_{n}^{2}+v_{n}^{2})\;\mathrm{d}x=0.
$$
It follows from Lemma~\ref{paper1lions} that $u_{n}\rightarrow0$ strongly in $L^{p}(\mathbb{R}^{N})$, for $2<p<2^{*}$. Notice that
$$
I_{\mu,\mathrm{o}}(u_{n},v_{n})-\frac{1}{2}\langle I_{\mu,\mathrm{o}}^{\prime}(u_{n},v_{n}),(u_{n},v_{n})\rangle=\frac{p-2}{2p}\mu\|u_{n}\|_{p}^{p}+\frac{1}{N}\|v_{n}\|_{2^{*}}^{2^{*}},
$$
which together with \eqref{paper1j6} and Lemma~\ref{paper1lions} implies that
$$
Nc_{\mathcal{N}_{\mu,\mathrm{o}}}+o_{n}(1)=N\left(I_{\mu,\mathrm{o}}(u_{n},v_{n})-\frac{1}{2}\langle I_{\mu,\mathrm{o}}^{\prime}(u_{n},v_{n}),(u_{n},v_{n})\rangle-\frac{p-2}{2p}\mu\|u_{n}\|_{p}^{p}\right)=\|v_{n}\|_{2^{*}}^{2^{*}}.
$$
Moreover, we can deduce that
$$
Nc_{\mathcal{N}_{\mu,\mathrm{o}}}+o_{n}(1)=\|v_{n}\|_{2^{*}}^{2^{*}}+\mu\|u_{n}\|_{p}^{p}+\langle I_{\mu,\mathrm{o}}^{\prime}(u_{n},v_{n}),(u_{n},v_{n})\rangle=\|(u_{n},v_{n})\|_{E_{\mathrm{o}}}^{2}-2\int_{\mathbb{R}^{N}}\lambda_{\mathrm{o}}(x)u_{n}v_{n}\;\mathrm{d}x.
$$
The preceding computations implies that 
$$
Nc_{\mathcal{N}_{\mu,\mathrm{o}}}+o_{n}(1)=\|v_{n}\|_{2^{*}}^{2^{*}}\leq S^{-\frac{N}{N-2}}\|\nabla v_{n}\|_{2}^{\frac{2N}{N-2}}\leq S^{-\frac{N}{N-2}}\left(\|(u_{n},v_{n})\|_{E_{\mathrm{o}}}^{2}-2\int_{\mathbb{R}^{N}}\lambda_{\mathrm{o}}(x)u_{n}v_{n}\;\mathrm{d}x\right)^{\frac{N}{N-2}}.
$$
Thus, we can conclude that
$$
Nc_{\mathcal{N}_{\mu,\mathrm{o}}}+o_{n}(1)\leq \left(\frac{Nc_{\mathcal{N}_{\mu,\mathrm{o}}}}{S}\right)^{\frac{N}{N-2}}+o_{n}(1).
$$
Therefore, $c_{\mathcal{N}_{\mu,\mathrm{o}}}\geq \frac{1}{N}S^{N/2}$, contradicting Lemma~\ref{paper1mu}. 

Since \eqref{paper1j500} holds, we can consider the shift sequence $(\tilde{u}_{n}(x),\tilde{v}_{n}(x))=(u_{n}(x+y_{n}),v_{n}(x+y_{n}))$ and we can repeat the same arguments used in the proof of Theorem~\ref{paper1A} to finish the proof.

\begin{remark}
	Let us set
	$
	\Lambda:=\{\mu>0:\mbox{\eqref{paper1j0}} \ \mbox{has ground state}\}.
	$
	We have proved in Theorem~\ref{paper1B} that $\Lambda$ is nonempty. Naturally arise the following questions: $\tilde{\mu}:=\inf\Lambda>0$? $\Lambda$ is an interval? Can we use the approach to study the existence of ground states for the system of the form:
	\begin{equation}\label{ss}
		\left\{
		\begin{array}{lr}
			-\Delta u+V_{1,\mathrm{o}}(x)u=|u|^{p-2}u+\lambda(x)v, & \quad x\in\mathbb{R}^{N},\\
			-\Delta v+V_{2,\mathrm{o}}(x)v=\mu|v|^{2^{*}-2}v+\lambda(x)u,    & x\in\mathbb{R}^{N}.
		\end{array}
		\right. \tag{$S_{\mu}$}
	\end{equation} 
	Does System~\eqref{ss} possesses ground state solution for any $\mu>0$?
\end{remark}


\section{Proof of Theorem \ref{paper1AB}}\label{paper1s6}

In this section we will be concerned with the existence of ground states for the asymptotically periodic case. We emphasize that the only difference between the potentials $V_{i,\mathrm{o}}(x),\lambda_{\mathrm{o}}(x)$ and $V_{i}(x),\lambda(x)$ is the periodicity required to $V_{i,\mathrm{o}}(x)$ and $\lambda_{\mathrm{o}}(x)$. Thus, if $V_{i}(x)$ and $\lambda(x)$ are periodic potentials, we can make use of Theorems \ref{paper1A} and \ref{paper1B} to get a ground state solution for System~\eqref{paper1jap}. Let us suppose that they are not periodic. 

Associated to System~\eqref{paper1jap}, we have the following energy functional 
$$
I_{\mu}(u,v)=\frac{1}{2}\left(\|(u,v)\|_{E}^{2}-2\int_{\mathbb{R}^{N}}\lambda(x)uv\;\mathrm{d}x\right)-\frac{\mu}{p}\|u\|_{p}^{p}-\frac{1}{q}\|v\|_{q}^{q}.
$$

\noindent The Nehari manifold associated to System~\eqref{paper1jap} is defined by
$$
\mathcal{N}_{\mu}=\{(u,v)\in E\backslash\{(0,0)\}:\langle I_{\mu}^{\prime}(u,v),(u,v)\rangle=0\},
$$
and the Nehari energy level is given by $c_{\mathcal{N}_{\mu}}=\inf_{\mathcal{N}_{\mu}} I_{\mu}(u,v)$. Arguing as before, we deduce that
$$   
I_{\mu}(u,v)\geq\left(\frac{1}{2}-\frac{1}{p}\right)(1-\delta)\|(u,v)\|_{E}^{2}\geq \left(\frac{1}{2}-\frac{1}{p}\right)(1-\delta)\alpha>0, \quad \mbox{for all} \hspace{0,2cm} (u,v)\in\mathcal{N}_{\mu}.
$$ 
Hence, $c_{\mathcal{N}_{\mu}}>0$. The next step is to establish a relation between the energy levels $c_{\mathcal{N}_{\mu,\mathrm{o}}}$ and $c_{\mathcal{N}_{\mu}}$. 

\begin{lemma}\label{paper1est}
	$c_{\mathcal{N}_{\mu}}<c_{\mathcal{N}_{\mu,\mathrm{o}}}$.
\end{lemma}
\begin{proof}
	Let $(u_{0},v_{0})\in \mathcal{N}_{\mu,\mathrm{o}}$ be the nonnegative ground state solution for System~\eqref{paper1j0}. It is easy to see that Lemma~\ref{paper1p1} works for $ I_{\mu}$ and $\mathcal{N}_{\mu}$. Thus, there exists a unique $t_{\mu}>0$, depending on $\mu$ and $(u_{0},v_{0})$, such that $(t_{\mu}u_{0},t_{\mu}v_{0})\in\mathcal{N}_{\mu}$. By using \ref{paper1A6} we get
	\begin{align*} 
		\int_{\mathbb{R}^{N}}\left[(V_{1}(x)-V_{1,\mathrm{o}}(x))u_{0}^{2}+(V_{2}(x)-V_{2,\mathrm{o}}(x))v_{0}^{2}+(\lambda_{\mathrm{o}}(x)-\lambda(x))u_{0}v_{0}\right]\mathrm{d}x<0.
	\end{align*}
	Therefore, $ I_{\mu}(t_{\mu}u_{0},t_{\mu}v_{0})-I_{\mu,\mathrm{o}}(t_{\mu}u_{0},t_{\mu}v_{0})<0$. Since $(u_{0},v_{0})$ is a ground state for System~\eqref{paper1j0} we can use Lemma~\ref{paper1p1} to deduce that 
	$$
	c_{\mathcal{N}_{\mu}}\leq I_{\mu}(t_{\mu}u_{0},t_{\mu}v_{0})<I_{\mu,\mathrm{o}}(t_{\mu}u_{0},t_{\mu}v_{0})\leq\max_{t\geq0}I_{\mu,\mathrm{o}}(tu_{0},tv_{0})=I_{\mu,\mathrm{o}}(u_{0},v_{0})=c_{\mathcal{N}_{\mu,\mathrm{o}}},
	$$   
	which finishes the proof.
\end{proof}     

Let $(u_{n},v_{n})_{n}\subset\mathcal{N}_{\mu}$ be the minimizing sequence satisfying
\begin{equation}\label{paper1paper4jj21}
	I_{\mu}(u_{n},v_{n})\rightarrow c_{\mathcal{N}_{\mu}} \quad \mbox{and} \quad   I_{\mu}^{\prime}(u_{n},v_{n})\rightarrow0.
\end{equation}    	
Since $(u_{n},v_{n})_{n}$ is a bounded sequence in $E$, we may assume up to a subsequence that $(u_{n},v_{n})\rightharpoonup(u_{0},v_{0})$ weakly in $E$. The main difficulty here is to prove that the weak limit is nontrivial. 

\begin{proposition}\label{paper1paper4p3}
	The weak limit $(u_{0},v_{0})$ of the minimizing sequence $(u_{n},v_{n})_{n}$ is nontrivial.
\end{proposition}
\begin{proof}
	We suppose by contradiction that $(u_{0},v_{0})=(0,0)$. We may assume that 
	\begin{itemize}
		\item $u_{n}\rightarrow 0$ and $v_{n}\rightarrow 0$ strongly in $L^{p}_{loc}(\mathbb{R}^{N})$, for all $2\leq p<2^{*}$;
		\item $u_{n}(x)\rightarrow 0$ and $v_{n}(x)\rightarrow 0$ almost everywhere in $\mathbb{R}^{N}$.
	\end{itemize}       
	It follows from assumption \ref{paper1A6} that for any $\varepsilon>0$ there exists $R>0$ such that
	\begin{equation}\label{paper1paper4jj2}
		|V_{1,\mathrm{o}}(x)-V_{1}(x)|<\varepsilon, \quad |V_{2,\mathrm{o}}(x)-V_{2}(x)|<\varepsilon, \quad |\lambda(x)-\lambda_{\mathrm{o}}(x)|<\varepsilon, \quad \mbox{for} \hspace{0,2cm} |x|\geq R.
	\end{equation}
	By using \eqref{paper1paper4jj2} and the local convergence, for any $\varepsilon>0$, there exists $n_{0}\in\mathbb{N}$ such that 
	$$
	\left|\int_{\mathbb{R}^{N}}(V_{1,\mathrm{o}}(x)-V_{1}(x))u_{n}^{2}\;\mathrm{d}x\right| \leq (\|V_{1}\|_{L^{\infty}(B_{R}(0))}+\|V_{1,\mathrm{o}}\|_{L^{\infty}(B_{R}(0))})\varepsilon + C\varepsilon,
	$$                                              
	$$
	\left|\int_{\mathbb{R}^{N}}(V_{2,\mathrm{o}}(x)-V_{2}(x))v_{n}^{2}\;\mathrm{d}x\right|\leq (\|V_{2,\mathrm{o}}\|_{L^{\infty}(B_{R}(0))}+\|V_{2}\|_{L^{\infty}(B_{R}(0))})\varepsilon + C\varepsilon,
	$$
	$$
	\hspace{-0,6cm}	\left|\int_{\mathbb{R}^{N}}(\lambda(x)-\lambda_{\mathrm{o}}(x))u_{n}v_{n}\;\mathrm{d}x\right| \leq (\|\lambda\|_{L^{\infty}(B_{R}(0))}+\|\lambda_{\mathrm{o}}\|_{L^{\infty}_{loc}(B_{R}(0))})\varepsilon + C\varepsilon,
	$$
	for all $n\geq\tilde{n_{0}}$. Therefore, we can conclude that
	$$
	I_{\mu,\mathrm{o}}(u_{n},v_{n})- I_{\mu}(u_{n},v_{n})=o_{n}(1) \quad \mbox{and} \quad \langle I_{\mu,\mathrm{o}}^{\prime}(u_{n},v_{n}),(u_{n},v_{n})\rangle-\langle I_{\mu}^{\prime}(u_{n},v_{n}),(u_{n},v_{n})\rangle=o_{n}(1),
	$$
	which jointly with \eqref{paper1paper4jj21} implies that
	\begin{equation}\label{paper1paper4jj22}
		I_{\mu,\mathrm{o}}(u_{n},v_{n})=c_{\mathcal{N}_{\mu}}+o_{n}(1) \quad \mbox{and} \quad \langle I_{\mu,\mathrm{o}}^{\prime}(u_{n},v_{n}),(u_{n},v_{n})\rangle=o_{n}(1).
	\end{equation}
	By using Lemma~\ref{paper1p1} we obtain a sequence $(t_{n})_{n}\subset(0,+\infty)$ such that $(t_{n}u_{n},t_{n}v_{n})_{n}\subset\mathcal{N}_{\mu,\mathrm{o}}$.
	
	\vspace{0,3cm}
	
	\noindent\textit{Claim 1.} $\limsup_{n\rightarrow+\infty}t_{n}\leq 1$.
	
	\vspace{0,3cm}
	
	Arguing by contradiction, we suppose that there exists $\varepsilon_{0}>0$ such that, up to a subsequence, we have $t_{n}\geq1+\varepsilon_{0}$, for all $n\in\mathbb{N}$. Thus, using \eqref{paper1paper4jj22} and the fact that $(t_{n}u_{n},t_{n}v_{n})\subset\mathcal{N}_{\mu,\mathrm{o}}$ we get
	$$
	(t_{n}^{p-2}-1)\mu\|u_{n}\|_{p}^{p}+(t_{n}^{q-2}-1)\|v_{n}\|_{q}^{q}=o_{n}(1),
	$$
	which together with $t_{n}\geq1+\varepsilon_{0}$ implies that
	\begin{equation}\label{paper1jj24}
		((1+\varepsilon_{0})^{p-2}-1)\mu\|u_{n}\|_{p}^{p}+((1+\varepsilon_{0})^{q-2}-1)\|v_{n}\|_{q}^{q}\leq o_{n}(1).
	\end{equation}
	Similarly to the proof of Theorems~\ref{paper1A} and \ref{paper1B}, we define $(\tilde{u}_{n}(x),\tilde{v}_{n}(x))=(u_{n}(x+y_{n}),v_{n}(x+y_{n}))$. It follows from assumption \ref{paper1A6} that $V_{1},V_{2}\in L^{\infty}(\mathbb{R}^{N})$. Using the continuous embedding $E_{i}\hookrightarrow H^{1}(\mathbb{R}^{N})$ we can deduce that $(\tilde{u}_{n},\tilde{v}_{n})_{n}$ is bounded in $E$. Thus, up to a subsequence, we may consider $(\tilde{u}_{n},\tilde{v}_{n})\rightharpoonup(\tilde{u},\tilde{v})$ weakly in $E$. Therefore,
	\begin{equation}\label{paper1j111} 
		\lim_{n\rightarrow+\infty}\int_{B_{R}(0)}(\tilde{u}_{n}^{2}+\tilde{v}_{n}^{2})\;\mathrm{d}x=\lim_{n\rightarrow+\infty}\int_{B_{R}(y_{n})}(u_{n}^{2}+v_{n}^{2})\;\mathrm{d}x\geq\beta>0,
	\end{equation}
	which implies $(\tilde{u},\tilde{v})\neq(0,0)$. We point out that in the critical case, when $q=2^{*}$, \eqref{paper1j111} holds for parameters $\mu\geq\mu_{0}$, where $\mu_{0}$ was introduced in Lemma~\ref{paper1mu}. Thus, by using \eqref{paper1jj24} and the semicontinuity of the norm, we get
	$$
	0<((1+\varepsilon_{0})^{p-2}-1)\mu\|\tilde{u}\|_{p}^{p}+((1+\varepsilon_{0})^{q-2}-1)\|\tilde{v}\|_{q}^{q}\leq o_{n}(1),
	$$
	which is not possible and finishes the proof of \textit{Claim 1.}
	
	\vspace{0,3cm}
	
	\noindent\textit{Claim 2.} \textit{There exists $n_{0}\in\mathbb{N}$ such that $t_{n}\geq1$, for $n\geq n_{0}$.}
	
	\vspace{0,3cm}
	
	In fact, arguing by contradiction, we suppose that up to a subsequence, $t_{n}<1$. Since $(t_{n}u_{n},t_{n}v_{n})_{n}\subset\mathcal{N}_{\mu,\mathrm{o}}$ we have that
	$$
	c_{\mathcal{N}_{\mu,\mathrm{o}}} \leq \frac{p-2}{2p}\mu t_{n}^{p}\|u_{n}\|_{p}^{p}+\frac{q-2}{2q}t_{n}^{q}\|v\|_{q}^{q}\leq \frac{p-2}{2p}\mu \|u_{n}\|_{p}^{p}+\frac{q-2}{2q}\|v\|_{q}^{q}=c_{\mathcal{N}_{\mu}}+o_{n}(1).
	$$
	Therefore, $c_{\mathcal{N}_{\mu,\mathrm{o}}}\leq c_{\mathcal{N}_{\mu}}$ which contradicts Lemma~\ref{paper1est} and finishes the proof of \textit{Claim 2.}
	
	\vspace{0,3cm}
	
	Combining \textit{Claims} $1$ and $2$ we deduce that 
	$$
	I_{\mu,\mathrm{o}}(t_{n}u_{n},t_{n}v_{n})-I_{\mu,\mathrm{o}}(u_{n},v_{n})=o_{n}(1).
	$$
	Thus, it follows from \eqref{paper1paper4jj22} that
	$$
	c_{\mathcal{N}_{\mu,\mathrm{o}}}\leq I_{\mu,\mathrm{o}}(t_{n}u_{n},t_{n}v_{n})=I_{\mu,\mathrm{o}}(u_{n},v_{n})+o_{n}(1)=c_{\mathcal{N}_{\mu}}+o_{n}(1),
	$$
	which contradicts Lemma~\ref{paper1est}. Therefore, $(u_{0},v_{0})\neq(0,0)$.
\end{proof}

\begin{proof}[Proof of Theorem~\ref{paper1AB} completed]
	Since $(u_{0},v_{0})$ is a nontrivial point of the energy functional $ I$, it follows that $(u_{0},v_{0})\in\mathcal{N}_{\mu}$. Therefore, we have $c_{\mathcal{N}_{\mu}}\leq  I_{\mu}(u_{0},v_{0})$. On the other hand, using the semicontinuity of the norm we deduce that	
	\begin{eqnarray*}
		c_{\mathcal{N}_{\mu}}+o_{n}(1) & = & \left(\frac{1}{2}-\frac{1}{p}\right)\mu\|u_{n}\|_{p}^{p}+\left(\frac{1}{2}-\frac{1}{q}\right)\|v_{n}\|_{q}^{q}\\
		& \geq & \left(\frac{1}{2}-\frac{1}{p}\right)\mu\|u_{0}\|_{p}^{p}+\left(\frac{1}{2}-\frac{1}{q}\right)\|v_{0}\|_{q}^{q}+o_{n}(1)\\
		& = &   I_{\mu}(u_{0},v_{0})+o_{n}(1).
	\end{eqnarray*}  
	Hence, $c_{\mathcal{N}_{\mu}}\geq  I_{\mu}(u_{0},v_{0})$. Therefore $ I_{\mu}(u_{0},v_{0})=c_{\mathcal{N}_{\mu}}$. Repeating the same argument used in the proof of Theorem~\ref{paper1A}, we can deduce that there exists $t_{\mu}>0$ such that $(t_{\mu}|u_{0}|,t_{\mu}|v_{0}|)\in\mathcal{N}_{\mu}$ is a positive ground state solution for System~\eqref{paper1jap} which finishes the proof of Theorem~\ref{paper1AB}.    
\end{proof}


\section{Proof of Theorem \ref{paper1C}}\label{paper1s4}

In this section we deal of the following coupled system
\begin{equation}\label{paper1j000}
	\left\{
	\begin{array}{lr}
		-\Delta u+V_{1}(x)u=\mu|u|^{2^{*}-2}u+\lambda(x)v, & \quad x\in\mathbb{R}^{N},\\
		-\Delta v+V_{2}(x)v=|v|^{2^{*}-2}v+\lambda(x)u,    & x\in\mathbb{R}^{N}.
	\end{array}
	\right.
\end{equation}
In order to obtain a nonexistence result we prove the following Pohozaev identity.

\begin{lemma}\label{paper1poho} Suppose $N\geq3$ and let $(u,v)\in E$ be a classical solution of \eqref{paper1j000}. Then, $(u,v)$ satisfies the following Pohozaev identity:
	\begin{align*}\label{paper1po1}
		\int_{\mathbb{R}^{N}}\left(|\nabla u|^{2}+|\nabla v|^{2}\right)\; \mathrm{d}x = \int_{\mathbb{R}^{N}}\left(\mu|u|^{2^{*}}+|v|^{2^{*}}+2^{*}\lambda(x)uv\right)\; \mathrm{d}x+\frac{2}{N-2}\int_{\mathbb{R}^{N}}\langle\nabla\lambda(x),x\rangle uv\; \mathrm{d}x\\-\frac{2^{*}}{2}\int_{\mathbb{R}^{N}}\left(V_{1}(x)u^{2}+V_{2}(x)v^{2}\right)\; \mathrm{d}x 
		-\frac{1}{N-2}\int_{\mathbb{R}^{N}}\left(\langle\nabla V_{1}(x),x\rangle u^{2}+\langle\nabla V_{2}(x),x\rangle v^{2}\right)\; \mathrm{d}x.
	\end{align*}
\end{lemma}
\begin{proof}
	In order to get this Pohozaev identity we adapt some ideas from \cite[Theorem B.3]{will}. Let $(u,v)\in E$ be a classical solution of the system \eqref{paper1j000} and let us denote
	$$
	f(x,u,v)=-V_{1}(x)u+\mu|u|^{2^{*}-2}u+\lambda(x)v \quad \mbox{and} \quad g(x,u,v)=-V_{2}(x)v+|v|^{2^{*}-2}v+\lambda(x)u.
	$$
	We consider the cut-off function $\psi\in C^{\infty}_{0}(\mathbb{R})$ defined by $\psi(t)=1$ if $|t|\leq 1$, $\psi(t)=0$ if $|t|\geq 2$ and $|\psi^{\prime}(t)|\leq C$, for some $C>0$. We define $\psi_{n}(x)=\psi\left(|x|^{2}/n^{2}\right)$ and we note that
	$$
	\nabla\psi_{n}(x)=\frac{2}{n^{2}}\psi^{\prime}\left(\frac{|x|^{2}}{n^{2}}\right)x.
	$$
	Multiplying the first equation in \eqref{paper1j000} by the factor $\langle\nabla u,x\rangle\psi_{n}$, the second equation by the factor $\langle\nabla v,x\rangle\psi_{n}$, summing and integrating we get
	\begin{equation}\label{paper1pohoz}
		\int_{\mathbb{R}^{N}}(\Delta u\langle\nabla u,x\rangle+\Delta v\langle\nabla v,x\rangle)\psi_{n}\;\mathrm{d}x=\int_{\mathbb{R}^{N}}(f(x,u,v)\langle\nabla u,x\rangle+g(x,u,v)\langle\nabla v,x\rangle)\psi_{n}\;\mathrm{d}x
	\end{equation}
	The idea is to take the limit as $n\rightarrow+\infty$ in \eqref{paper1pohoz}. In order to calculate the limit in the left-hand side of \eqref{paper1pohoz}, we note that
	\begin{equation}\label{paper1p6}
		\langle\nabla u,x\rangle\psi_{n}\Delta u = \mbox{div}(\psi_{n}H(x,u))+\frac{N-2}{2}\psi_{n}|\nabla u|^{2}+\frac{|\nabla u|^{2}}{2}\langle\nabla\psi_{n},x\rangle-\langle\nabla u,x\rangle\langle\nabla\psi_{n},\nabla u\rangle,
	\end{equation}
	where
	$
	H(x,u)=\langle\nabla u,x\rangle\nabla u-(|\nabla u|^{2}/2)x.
	$
	Therefore, integrating \eqref{paper1p6} and using Lebesgue dominated convergence theorem, we conclude that
	\begin{equation}\label{paper1p7}
		-\lim_{n\rightarrow\infty}\int_{\mathbb{R}^{N}}\langle\nabla u,x\rangle\psi_{n}\Delta u \;\mathrm{d}x=-\frac{N-2}{2}\int_{\mathbb{R}^{N}}|\nabla u|^{2}\; \mathrm{d}x.
	\end{equation}
	Analogously, we can deduce the limit 
	\begin{equation}\label{paper1p10}
		-\lim_{n\rightarrow\infty}\int_{\mathbb{R}^{N}}\langle\nabla v,x\rangle\psi_{n}\Delta v \;\mathrm{d}x=-\frac{N-2}{2}\int_{\mathbb{R}^{N}}|\nabla v|^{2}\; \mathrm{d}x.
	\end{equation}
	In order to calculate the right-hand side, we note that
	$$
	\mbox{div}\left(\psi_{n}F(x,u,v)x\right) =\psi_{n}\langle\nabla F(x,u,v),x\rangle+F(x,u,v)\langle\nabla\psi_{n},x\rangle+N\psi_{n}F(x,u,v),
	$$
	where
	$
	F(x,u,v)=-(1/2)V_{1}(x)u^{2}+(\mu/2^{*})|u|^{2^{*}}+\lambda(x)uv.
	$
	Hence, we can deduce that
	\begin{align*}\label{paper1p8}
		\int_{\mathbb{R}^{N}}f(x,u,v)\langle\nabla u,x\rangle\psi_{n}\; \mathrm{d}x = \int_{\mathbb{R}^{N}}\left(\mbox{div}(\psi_{n}F(x,u,v)x)-F(x,u,v)\langle\nabla\psi_{n},x\rangle\psi_{n}\right)\; \mathrm{d}x\nonumber\\ +\int_{\mathbb{R}^{N}}\left(\frac{1}{2}\langle\nabla V_{1}(x),u\rangle u^{2}-NF(x,u,v)\psi_{n}-\langle\nabla\lambda(x),x\rangle uv-\langle\lambda(x)u\nabla v,x\rangle\right)\psi_{n}\; \mathrm{d}x.
	\end{align*}
	Analogously, denoting
	$
	G(x,u,v)=-\frac{1}{2}V_{2}(x)v^{2}+\frac{1}{2^{*}}|v|^{2^{*}}+\lambda(x)uv,
	$
	we can deduce that
	\begin{align*}
		\int_{\mathbb{R}^{N}}g(x,u,v)\langle\nabla v,x\rangle\psi_{n}\; \mathrm{d}x = \int_{\mathbb{R}^{N}}\left(\mbox{div}(\psi_{n}G(x,u,v)x)-G(x,u,v)\langle\nabla\psi_{n},x\rangle\psi_{n}\right)\; \mathrm{d}x\nonumber\\  +\int_{\mathbb{R}^{N}}\left(\frac{1}{2}\langle\nabla V_{2}(x),v\rangle v^{2}-NG(x,u,v)\psi_{n}-\langle\nabla\lambda(x),x\rangle uv-\langle\lambda(x)v\nabla u,x\rangle\right)\psi_{n}\; \mathrm{d}x.
	\end{align*}
	By using integration by parts we have that	
	$$
	-\int_{\mathbb{R}^{N}}\lambda(x)\langle u\nabla v+v\nabla u,x\rangle\psi_{n}\; \mathrm{d}x =\int_{B_{2n}(0)}\left(\langle\nabla\psi_{n},x\rangle\lambda(x)uv+\langle\nabla\lambda(x),x\rangle\psi_{n}uv+N\psi_{n}\lambda(x)uv\right)\; \mathrm{d}x,
	$$
	which implies that
	$$
	\lim_{n\rightarrow\infty}\int_{\mathbb{R}^{N}}\lambda(x)\langle u\nabla v+v\nabla u,x\rangle\psi_{n}\; \mathrm{d}x=-\int_{\mathbb{R}^{N}}\langle\nabla\lambda(x),x\rangle uv\; \mathrm{d}x-N\int_{\mathbb{R}^{N}}\lambda(x)uv\; \mathrm{d}x.
	$$
	Therefore, using the Lebesgue dominated convergence theorem in the same way as we used when we calculate the left-hand side, we obtain
	\begin{eqnarray*}
		\lim_{n\rightarrow\infty}\int_{\mathbb{R}^{N}}\left(f(x,u,v)\langle\nabla u,x\rangle+g(x,u,v)\langle\nabla v,x\rangle\right)\psi_{n}\; \mathrm{d}x & = & -N\int_{\mathbb{R}^{N}}\left(F(x,u,v)+G(x,u,v)\right)\; \mathrm{d}x +\nonumber\\ & & \hspace{-9,0cm}+\frac{1}{2}\int_{\mathbb{R}^{N}}\left(\langle\nabla V_{1}(x),x\rangle u^{2}+\langle\nabla V_{2}(x),x\rangle v^{2}\right)\; \mathrm{d}x-\int_{\mathbb{R}^{N}}\langle\nabla\lambda(x),x\rangle uv\; \mathrm{d}x+N\int_{\mathbb{R}^{N}}\lambda(x)uv \; \mathrm{d}x.
	\end{eqnarray*} 
	Replacing $F(x,u,v)$ and $G(x,u,v)$ in the equation above, we get the right-hand side of \eqref{paper1pohoz} which finishes the proof.
\end{proof}

\begin{proof}[Proof of Theorem~\ref{paper1C} completed]
	Let $(u,v)\in E$ be a positive classical solution of \eqref{paper1j000}. By the definition of weak solution we obtain
	\begin{equation}\label{paper1p14}
		\int_{\mathbb{R}^{N}}\left(|\nabla u|^{2}+V_{1}(x)u^{2}+|\nabla v|^{2}+V_{2}(x)v^{2}\right)\;\mathrm{d}x=\int_{\mathbb{R}^{N}}\left(\mu|u|^{2^{*}}+|v|^{2^{*}}+2\lambda(x)uv\right)\;\mathrm{d}x.
	\end{equation}
	Combining \eqref{paper1p14} with the Pohozaev identity obtained in Lemma~\ref{paper1poho}, we have 
	\begin{align}\label{paper1p15}
		0 = \left(1-\frac{2^{*}}{2}\right)\int_{\mathbb{R}^{N}}\left(V_{1}(x)u^{2}+V_{2}(x)v^{2}-2\lambda(x)uv\right)\; \mathrm{d}x+\frac{2}{N-2}\int_{\mathbb{R}^{N}}\langle\nabla\lambda(x),x\rangle uv\; \mathrm{d}x\nonumber\\ -\frac{1}{N-2}\int_{\mathbb{R}^{N}}\left(\langle\nabla V_{1}(x),x\rangle u^{2}+\langle\nabla V_{2}(x),x\rangle v^{2}\right)\; \mathrm{d}x.
	\end{align}
	Multiplying \eqref{paper1p15} by the factor $-(N-2)/2$, we get
	\begin{align*}
		\int_{\mathbb{R}^{N}}\left(V_{1}(x)u^{2}+V_{2}(x)v^{2}-2\lambda(x)uv\right)\; \mathrm{d}x =\int_{\mathbb{R}^{N}}\langle\nabla\lambda(x),x\rangle uv\; \mathrm{d}x\nonumber\\-\frac{1}{2}\int_{\mathbb{R}^{N}}\left(\langle\nabla V_{1}(x),x\rangle u^{2}+\langle\nabla V_{2}(x),x\rangle v^{2}\right)\; \mathrm{d}x.
	\end{align*}
	Thus, it follows from assumptions \ref{paper1A4} and \ref{paper1A5} that
	$$
	\int_{\mathbb{R}^{N}}\left(V_{1}(x)u^{2}+V_{2}(x)v^{2}-2\lambda(x)uv\right)\; \mathrm{d}x\leq0.
	$$
	On the other hand, by assumption \ref{paper1A3} we get
	$$
	\int_{\mathbb{R}^{N}}\left(V_{1}(x)u^{2}+V_{2}(x)v^{2}-2\lambda(x)uv\right)\; \mathrm{d}x\geq0.
	$$
	Thus, we conclude that
	$$
	\int_{\mathbb{R}^{N}}\left(V_{1}(x)u^{2}+V_{2}(x)v^{2}-2\lambda(x)uv\right)\; \mathrm{d}x=0.
	$$
	Therefore, we finally deduce that
	\begin{eqnarray*}
		0 & \leq & \int_{\mathbb{R}^{N}}\left(V_{1}(x)u^{2}-2\sqrt{V_{1}(x)V_{2}(x)}uv+V_{2}(x)v^{2}\right)\;\mathrm{d}x\\
		& \leq & \int_{\mathbb{R}^{N}}\left(V_{1}(x)u^{2}+V_{2}(x)v^{2}-\frac{2}{\delta}\lambda(x)uv\right)\;\mathrm{d}x\\
		& < & \int_{\mathbb{R}^{N}}\left(V_{1}(x)u^{2}+V_{2}(x)v^{2}-2\lambda(x)uv\right)\;\mathrm{d}x=0,
	\end{eqnarray*}
	which is a contradiction and this finishes the proof of Theorem \ref{paper1C}.
\end{proof}

{\bf Acknowledgements.} The authors would like to express their sincere gratitude to the referee for carefully reading the manuscript and valuable comments and suggestions.









\nocite{*} 

%


\end{document}